\documentclass{amsart}
\usepackage{amssymb}
\usepackage{amsfonts}
\usepackage{amssymb}
\usepackage{amsmath}
\usepackage{amsthm}
\usepackage{enumerate}
\usepackage{tabularx}
\usepackage{centernot}
\usepackage{mathtools}
\usepackage{stmaryrd}
\usepackage{amsthm,amssymb}
\usepackage{etoolbox}
\usepackage{tikz}
\usepackage{amssymb}
\usetikzlibrary{matrix}
\usepackage{tikz-cd}
\usepackage{tikz}
\usepackage{marginnote}
\definecolor{mygray}{gray}{0.85}
\usepackage[backgroundcolor=mygray,colorinlistoftodos,prependcaption,textsize=small]{todonotes}
\usepackage{xargs}

\renewcommand{\leq}{\leqslant}
\renewcommand{\geq}{\geqslant}

\makeatletter
\def\subsection{\@startsection{subsection}{3}%
  \z@{.5\linespacing\@plus.7\linespacing}{.3\linespacing}%
  {\bfseries\centering}}
\makeatother

\makeatletter
\def\subsubsection{\@startsection{subsubsection}{3}%
  \z@{.5\linespacing\@plus.7\linespacing}{.3\linespacing}%
  {\centering}}
\makeatother

\makeatletter
\def\myfnt{\ifx\protect\@typeset@protect\expandafter\footnote\else\expandafter\@gobble\fi}
\makeatother

\newtheorem{theorem}{Theorem}

\newtheorem{corollary}[theorem]{Corollary}
\newtheorem{definition}[theorem]{Definition}

\newtheorem{proposition}[theorem]{Proposition}

\newtheorem{fact}[theorem]{Fact}

\newtheorem*{nfact}{Fact}

\newcounter{claimcounter}

\newcommand{\pureindep}[1][]{%
  \mathrel{
    \mathop{
      \vcenter{
        \hbox{\oalign{\noalign{\kern-.3ex}\hfil$\vert$\hfil\cr
              \noalign{\kern-.7ex}
              $\smile$\cr\noalign{\kern-.3ex}}}
      }
    }\displaylimits_{#1}
  }
}

\newcommand{\indep}[2]{%
  \mathrel{
    \mathop{
      \vcenter{
        \hbox{%
\oalign{
\noalign{\kern-.3ex}\hfil$\vert$\hfil\cr
              \noalign{\kern-.7ex}
              $\smile$\cr\noalign{\kern-.3ex}
}
}
      }
}^{\!\!\!\!\!#2}_{\!\!\hspace{-0.1em}#1}
  }
}

\newcommand{\displayindep}[2]{%
  \mathrel{
    \mathop{
      \vcenter{
        \hbox{%
\oalign{
\noalign{\kern-.3ex}\hfil$\vert$\hfil\cr
              \noalign{\kern-.7ex}
              $\smile$\cr\noalign{\kern-.3ex}
}
}
      }
}^{\!\!\hspace{-0.1em}#2}_{\!\!\hspace{-0.1em}#1}
  }
}

\newcommand{\displayfindep}[2]{%
  \mathrel{
    \mathop{
      \vcenter{
        \hbox{%
\oalign{
\noalign{\kern-.3ex}\hfil$\vert$\hfil\cr
              \noalign{\kern-.7ex}
              $\smile$\cr\noalign{\kern-.3ex} 
}
}
      }
}^{\!\hspace{-0.14em}#2}_{\!\!\hspace{-0.05em}#1}
  }
}

\begin{document}

\begin{abstract} We show that in algebraically locally finite countable homogeneous structures with a free stationary independence relation the small index property implies the strong small index property. We use this and  the main result of \cite{siniora} to deduce that countable free homogeneous structures in a locally finite relational language have the strong small index property. We also exhibit new continuum sized classes of $\aleph_0$-categorical structures with the strong small index property whose automorphism groups are pairwise non-isomorphic.
\end{abstract}

\title[The SSIP for Free Homogeneous Structures]{The Strong Small Index Property for Free Homogeneous Structures}
\thanks{Partially supported by European Research Council grant 338821. No. 1108 on Shelah's publication list.}

\author{Gianluca Paolini}
\address{Einstein Institute of Mathematics,  The Hebrew University of Jerusalem, Israel}

\author{Saharon Shelah}
\address{Einstein Institute of Mathematics,  The Hebrew University of Jerusalem, Israel \and Department of Mathematics,  Rutgers University, U.S.A.}

\date{\today}
\maketitle


\section{Introduction}

	The {\em small index property} (SIP) of a countable structure $M$ asserts that any subgroup of $Aut(M)$ of index $< 2^{\aleph_0}$ contains the pointwise \mbox{stabilizer of a finite $A \subseteq M$.}
	
	The small index property received remarkable attention from the logic community over the years, for a nice and thorough overview see \cite[Section 4.2.6]{hodges}. The most fundamental consequence of SIP is that this property allows to recover the topological structure of $Aut(M)$ from the abstract group structure of $Aut(M)$, since a countable structure with SIP is such that the open subgroups of $Aut(M)$ are exactly the subgroups of small index (i.e. of index $<2^{\aleph_0}$). A classical result related to the power of SIP to ``remember'' model-theoretic properties is the following:
	
	\begin{nfact}[Lascar \cite{lascar}] Let $M$ and $N$ be countable $\aleph_0$-categorical structures and suppose that $M$ has SIP. Then $Aut(M) \cong Aut(N)$ iff $M$ and $N$ are bi-interpretable.
\end{nfact}
	
	A variant of SIP which is less known (but very powerful) is the {\em strong} small index property (SSIP), which requires that every subgroup of $Aut(M)$ of index less than $2^{\aleph_0}$ lies between the pointwise  and the setwise stabilizer of a finite set $A \subseteq M$. This property allows for even finer reconstruction results (i.e. to ``remember'' more).
	
	For example, in \cite{cameron} Cameron proves that the random graph has the strong small index property and uses this to show that its group of automorphisms is not isomorphic to any group of automorphisms of a graph (other than the random graph) or digraph which is transitive on vertices, ordered edges, and ordered non-edges.
	Another remarkable example is the following reconstruction result recently proved by the authors (and inspired by a famous earlier work of Rubin \cite{rubin}):
		
	\begin{nfact}[\cite{Sh_Pa_recon}] Let $M$ and $N$ be countable $\aleph_0$-categorical structures with the strong small index property and no algebraicity. Then $ Aut(M)$ and $Aut(N)$ are isomorphic as abstract groups if and only if $M$ and $N$ are bi-definable. 
\end{nfact}

	Thus, the isolation of good sufficient conditions for SSIP is a valuable tool, which broadens the domain of application of the general results in reconstruction theory which assume SSIP. This is the topic of the present note. We prove here:

	\begin{theorem}\label{SIP_implies_SSIP} Let $M$ be an algebraically locally finite countable homogeneous structure which admits a free stationary independence relation (cf. Section~\ref{sec2}). If $M$ has the small index property, then $M$ has the strong small index property.
\end{theorem}

	We then use Theorem \ref{SIP_implies_SSIP} and the main result of \cite{siniora} to infer:

	\begin{corollary}\label{SSIP_free} Let $M$ be a countable free homogeneous structure in a locally finite irreflexive relational language. Then $M$ has the strong small index property.
\end{corollary}

	Thus deducing:

\begin{corollary} The following structures have the strong small index property:
	\begin{enumerate}[(1)]
	\item the random graph \cite{rado};
	\item the universal homogeneous $K_n$-free graph ($n \geq 3$) \cite{henson_graphs};
	\item the $n$-coloured random graph \cite{truss};
	\item the random directed graph;
	\item the continuum many Henson digraphs \cite{henson_digraphs};
	\item the $k$-uniform random hypergraph ($k \geq 2$) \cite{thomas};
	\item the random $m$-free $k$-uniform hypergraph ($m > k \geq 2$) \cite{hru}.
	\end{enumerate}
\end{corollary}
	
	Finally, we define what we call $\eta$-hypergraphs and random $\zeta$-free $\eta$-hypergraphs, for some $\eta \in 2^{\omega}$ and $\zeta \in \omega^\omega$, and use Corollary \ref{SSIP_free} and \cite{Sh_Pa_recon} to show:

\begin{theorem}\label{pairwise_non_iso} Let $\mathbf{K}$ be one of the following two classes of countable structures:
	\begin{enumerate}
	\item the random $\eta$-hypergraphs $M(\eta)$, for some $\eta \in 2^{\omega}$ (cf. Definition \ref{iper_1});
	\item the random $\zeta$-free $\eta$-hypergraphs $M(\eta, \zeta)$, for some $\eta \in 2^{\omega}$ and $\zeta \in \omega^\omega$ (cf. Definition \ref{iper_2}).
\end{enumerate} 
If $M = M(x)$, $N = N(y) \in \mathbf{K}$ and $x \neq y$, then $Aut(M) \not\cong Aut(N)$. Moreover, every $M \in \mathbf{K}$ has the strong small index property.
\end{theorem}

	In particular, we exhibit new continuum sized classes of homogeneous $\aleph_0$-categorical structures with the strong small index property whose automorphism groups are pairwise non-isomorphic (as in the case of the Henson digraphs \cite{henson_digraphs}).
	
	
\section{Definitions}\label{sec2}

	

	As a matter of notation, given a structure $M$ and $A \subseteq M$, and considering $Aut(M) = G$ in its natural action on $M$, we denote the pointwise (resp. setwise) stabilizer of $A$ under this action by $G_{(A)}$ (resp. $G_{\{ A \}}$).	
	
	We say that a countable structure $M$ is {\em homogeneous} if every isomorphism between finitely generated substructures of $M$ extends to an automorphism of $M$. As well known, a countable homogeneous structure $M$ is the so-called Fra\"iss\'e limit of its {\em age} $\mathbf{K} = \mathbf{K}(M)$, i.e. the collection of finitely generated substructures of $M$ (for details on these notions see e.g. \cite[Chapter 6]{hodges}). Ages of homogeneous structures are characterized by the following three conditions:
	
	\begin{enumerate}[(I)]
	\item If $B \in \mathbf{K}$ and $A$ is a substructure of $B$ (denoted $A \leq B$), then $A \in \mathbf{K}$ ({\em Hereditary Property}).
	\item For every ${B}_1, {B}_2 \in \mathbf{K}$ there are ${C} \in \mathbf{K}$ and embeddings $f_i: {B}_i \rightarrow {C}$ ($i = 1,2$) ({\em Joint Embedding Property}).
	\item For every ${A}, {B}_1, {B}_2 \in \mathbf{K}$ and embeddings $f_i: A \rightarrow B_i$ ($i = 1,2$), there are ${C} \in \mathbf{K}$ and embeddings $g_i : {B}_i \rightarrow {C}$ ($i = 1,2$) such that $g_1f_1 \restriction A = g_2f_2 \restriction A$ ({\em Amalgamation Property}).
\end{enumerate}
	
	\begin{definition}\label{def_amalgam}
	\begin{enumerate}[(1)]
	\item We refer to classes of finitely generated structures satisfying conditions (I)-(III) above as {\em amalgamation classes}.
	\item The structure $C$ in (III) above is called an {\em amalgam} of $B_1$ and $B_2$ over $A$.
	\item We say that the amalgamation from (III) is {\em disjoint} if in addition the $C$ and $g_i$ can be chosen so that $g_1(B_1) \cap g_2(B_2) = g_1f_1(A)$.
	\end{enumerate}
\end{definition}  

	We will mostly deal with countable homogeneous structures satisfying two additional conditions: {\em locally finite algebraicity} and {\em canonical amalgamation}.
	
	\begin{definition} Let $M$ be a countable structure and $G = Aut(M)$.
	\begin{enumerate}[(1)]
	\item We say that $a$ is algebraic over $A \subseteq M$ if the orbit of $a$ under $G_{(A)}$ is finite. 
	\item The {\em algebraic closure} of $A \subseteq M$ in $M$, denoted as $acl_M(A)$, is the set of elements of $M$ which are algebraic over $A$.
	\item We say that $M$ has {\em locally finite algebraicity} if for every finite $A \subseteq M$, $acl_M(A)$ is finite.
\end{enumerate}
\end{definition}

	Notice that in countable homogeneous structures with disjoint amalgamation (cf. Definition \ref{def_amalgam}) we have $acl_M(A) = \langle A \rangle_M$, i.e. the algebraic closure of $A$ in $M$ equals the substructure generated by $A$ in $M$ (see e.g. \cite[Theorem 7.1.8]{hodges}).

	\begin{definition}\label{can_amal} Let $M$ be an homogeneous structure and $\mathbf{K} = \mathbf{K}(M)$ its age. We say that $M$ has {\em canonical amalgamation} if there exists an operation $B_1 \oplus_A B_2$ on $\mathbf{K}^3$ satisfying the following conditions:
	\begin{enumerate}[(a)]
	\item $B_1 \oplus_A B_2$ is defined when ${A} \leq {B}_i$ ($i = 1,2$) and $B_1 \cap B_2 = A$;
	\item $B_1 \oplus_A B_2$ is an amalgam of $B_1$ and $B_2$ over $A$  (cf. Definition \ref{def_amalgam});
	\item if $B_1 \oplus_A B_2$ and $B'_1 \oplus_{A'} B'_2$ are defined, and there exist $f_i: B_i \cong B'_i$ ($i = 1,2$) with $f_1 \restriction A = f_2 \restriction A$, then there is:
	$$f: B_1 \oplus_A B_2 \cong B'_1 \oplus_{A'} B'_2$$
such that $f \restriction B_1 = f_1$ and $f \restriction B_2 = f_2$.
\end{enumerate}
\end{definition}

	We will use the following notation: given $A, B, C \leq M$ we write $A \equiv_B C$ to mean that there is an automorphism of $M$ fixing $B$ pointwise and mapping $A$~to~$C$. 
	
	\begin{definition}\label{indep_rel} Let $M$ be an homogeneous structure with canonical amalgamation. We define a ternary relation between finite substructures of $M$ by setting $A \pureindep[C] B$ if and only if $\langle A, B, C \rangle_M \cong \langle A, C \rangle_M \oplus_C \langle B, C \rangle_M$. 
\end{definition}
	
	In many cases the relation defined in Definition \ref{indep_rel} satisfies several properties of interest, but at this level of generality we only have three of these properties:

	\begin{proposition} Let $M$ be a countable homogeneous structure with canonical amalgamation. Then (recalling Def. \ref{indep_rel}) for $A, B, C$ finite substructures of $M$ we have:
\begin{enumerate}[$(A)$]
		\item (Invariance) if $A \pureindep[C] B$ and $f \in Aut(M)$, then $f(A) \pureindep[f(C)] f(B)$;
		\item (Existence) for every $A, B, C \leq M$, there exists $A' \equiv_B A$ such that $A' \pureindep[B] C$;
		\item (Stationarity) if $A \equiv_C A'$, $A \pureindep[C] B$ and $A' \pureindep[C] B$, then $A \equiv_{\langle BC \rangle} A'$.
\end{enumerate}
\end{proposition}

	As mentioned above, usually the relation defined in Definition~\ref{indep_rel} satisfies several other properties of interest, this inspires the following definition:
	
	\begin{definition}[\cite{tent_ziegler, muller}]\label{def_stationary} Let $M$ be a countable homogeneous structure with locally finite algebraicity. We say that the ternary relation $\pureindep$ between finite substructures of $M$ is a {\em stationary independence relation} if it satisfies:
	\begin{enumerate}[(1)]
	\item (Invariance) if $A \pureindep[C] B$ and $f \in Aut(M)$, then $f(A) \pureindep[f(C)] f(B)$;
	\item (Symmetry) if $A \pureindep[C] B$, then $B \pureindep[C] A$.
	\item (Monotonicity) If $A \pureindep[C] \langle BD \rangle$ and $A \pureindep[C] B$, then $A \pureindep[\langle BC \rangle] D$;
	\item (Existence) for every $A, B, C \leq M$, there exists $A' \equiv_B A$ such that $A' \pureindep[B] C$;
	\item (Stationarity) if $A \equiv_C A'$, $A \pureindep[C] B$ and $A' \pureindep[C] B$, then $A \equiv_{\langle BC \rangle} A'$.
\end{enumerate}
\end{definition}

	A property of interest which is less common (but crucial for us) is:
	
	\begin{definition}[\cite{conant}]\label{free_stationay} We say that the stationary independence relation $\pureindep$ is free if it satisfies the following additional condition:
	\begin{enumerate}[(6)]
	\item (Freeness) if $A \pureindep[C] B$ and $C \cap AB \subseteq D \subseteq C$, then $A \pureindep[D] B$.
\end{enumerate}
\end{definition}

	Before proving Theorem \ref{SIP_implies_SSIP} we need a fact from \cite{hodges}:

	\begin{fact}[{\cite[Theorem 4.2.9]{hodges}}]\label{fact_1} Let $M$ be a countable homogeneous structure, $G = Aut(M)$, and suppose that for every finite algebraically closed $A, B \leq M$ we have $G_{(A \cap B)} = \langle G_{(A)} \cup G_{(B)} \rangle_G$. Let $H$ be a subgroup of $G$ such that there is some finite algebraically closed set $A$ with $G_{(A)} \leq H$. Then there is a unique smallest algebraically closed set $B \leq A$ such that $G_{(B)} \leq H$. Moreover, \mbox{for this $B$, $H \leq G_{\{B\}}$.}
\end{fact}

	The following proof is an abstract version of \cite[Lemma 27]{tent_mac}.

	\begin{proof}[Proof of Theorem \ref{SIP_implies_SSIP}] By Fact \ref{fact_1} it suffices to prove that for finite algebraically closed $A, B \leq M$ we have: $$G_{(A \cap B)} = \langle G_{(A)} \cup G_{(B)} \rangle_G.$$
The containment from right to left is trivial, so let $g \in G_{(A \cap B)}$. By Existence (cf.~Definition~\ref{def_stationary}(4)) there $h_1 \in G_{(A)}$ with $h_1g(A) \pureindep[A] B$. Then, since $g, h_1 \in G_{(A \cap B)}$ we have that $h_1g(A) \cap B = A \cap B$. Now, again by Existence, there $h_2 \in G_{(B)}$ with $h_2h_1g(A) \pureindep[B] A$. Then, by the above and the fact that $h_2 \in G_{(A \cap B)}$, we have that $h_2h_1g(A) \cap B = A \cap B$. Hence, by Freeneness (cf.~Definition~\ref{free_stationay}), we have that $h_2h_1g(A) \pureindep[A \cap B] A$. Now, again by Existence, there is $h_3 \in G_{(B)}$ such that $h_3(A) \pureindep[B] A$. Then, since $h_3 \in G_{(A \cap B)}$ we have that $h_3(A) \cap A = A \cap B$ and so, by Freeness, we have that $h_3(A) \pureindep[A \cap B] A$. Notice now that $h_3(A) \equiv_{A \cap B} h_2h_1g(A)$, since $g, h_1, h_2, h_3 \in G_{(A \cap B)}$. Hence, by Stationarity (cf.~Definition~\ref{def_stationary}(5)), we can find $h_4 \in G_{(A)}$ such that $h_4h_2h_1g \restriction A = h_3 \restriction A$. Thus, $h_* := h_3^{-1}h_4h_2h_1g \in G_{(A)}$ and so:
$$g = h_1^{-1}h_2^{-1}h_4^{-1}h_3h_* \in \langle G_{(A)} \cup G_{(B)} \rangle_G,$$
since $h_1, h_4, h_* \in G_{(A)}$ and $h_2, h_3 \in G_{(B)}$.
\end{proof}

\section{Applications}

	Given a relational language $L$ and $L$-structures $A, B_1, B_2$, with $A \subseteq B_1, B_2$, there is a very natural way of amalgamating $B_1$ and $B_2$ over $A$: the structure with domain $B_1 \cup B_2$ where the only relations are the relations from $B_1$ and the relations from $B_2$. This way of amalgamating is known as {\em free amalgamation}. 
	
	\begin{definition} We say that a countable homogeneous relational structure $M$ is {\em free} if its age $\mathbf{K}(M)$ is closed under free amalgamation.
\end{definition}

	It is easy to see that free amalgamation is an instance of canonical amalgamation (in the sense of Definition~\ref{can_amal}). Furthermore:
	
	\begin{proposition}\label{prop_free} Let $M$ be a free homogeneous relational structure. Then  the relation $\pureindep$ from Definition~\ref{can_amal} is a free stationay independence relation (cf. Def.~\ref{free_stationay}).
\end{proposition}

	Using the following fact from \cite{siniora} it is now immediate to deduce Corollary \ref{SSIP_free} (for an overview of the notion occurring there see \cite[Section 6.2]{kechris}). Recall that a relational language $L$ is said to be {\em locally finite} if for every $n < \omega$ there are only finitely many relations of arity $n$ in $L$. With some abuse of notation, we say that a language $L$ is irreflexive if we only consider $L$-structures $M$ such that if $R \in L$ and $M \models R(a_1, ..., a_n)$, then the $a_i$ are distinct.

	\begin{fact}[\cite{siniora}]\label{fact_2} Let $\mathbf{K}$ be a free amalgamation class of finite structures in a finite relational language. Then $\mathbf{K}$ has the extension property for partial automorphisms.
\end{fact}

	\begin{proof}[Proof of Corollary \ref{SSIP_free}] Of course Fact \ref{fact_2} can be used to show the extension property for partial automorphisms also for free amalgamation classes of finite irreflexive structures in a locally finite relational language. Thus, by  \cite[Sections 6.1-2]{kechris} we have the small index property. Hence, by Theorem \ref{SIP_implies_SSIP} \mbox{and Proposition \ref{prop_free} we are done.}
\end{proof}



	We now define the class of hypergraphs that appear in Theorem \ref{pairwise_non_iso}.
	\begin{definition}\label{iper_1} Let $\eta \in 2^{\omega}$ and $L(\eta)$ be the relational language which has exactly one relation symbol $R_n$ of arity $n$ if and only if $\eta(n) = 1$. For non-triviality reasons we restrict to the class of $L(\eta)$ with $\eta(0) = \eta(1) = 0$. Let $\mathbf{K}(\eta)$ be the class of finite $L(\eta)$-structure such that the relations $R_n$ (for $\eta(n) = 1$) are symmetric and irreflexive, i.e. if $K \models R(a_1, ..., a_n)$ then the $a_1, ..., a_n$ are distinct and $R(a_1, ..., a_n)$ iff $R(a_{\sigma(1)}, ..., a_{\sigma(n)})$ for every $\sigma \in Sym(\{ 1, ..., n \})$. The class $\mathbf{K}(\eta)$ is an amalgamation class. We call its Fra\"iss\'e limit $M({\eta})$ the random $\eta$-hypergraph.
\end{definition}

	For $\eta$ such that $\eta(n) = 1$ iff $n = k$, for fixed $k \geq 2$, the structure $M({\eta})$ is simply the random $k$-uniform hypergraph \cite{thomas}. Notice that $\mathbf{K}(\eta)$ is closed under free amalgamation, and so by our Corollary \ref{SSIP_free} we have:
	
	\begin{corollary} For $\eta$ as above, $M({\eta})$ has the strong small index property.
\end{corollary}

		\begin{definition}\label{iper_2} Let $\eta$ and $L(\eta)$ be as in the previous paragraph. Let $\zeta \in \omega^\omega$ be such that $\zeta(n) \geq n$ if $\eta(n) = 1$, and $\zeta(n) = 0$ otherwise. Let now $\mathbf{K}(\eta, \zeta)$ be the class of structures $K \in \mathbf{K}(\eta)$ such that, for $n < \omega$ such that $\zeta(n) > n$, $K$ does not embed the structure of size $\zeta(n)$ such that every tuple of $n$ distinct elements is in the relation $R_n$. The class $\mathbf{K}(\eta, \zeta)$ is an amalgamation class. We call its Fra\"iss\'e limit $M(\eta, \zeta)$ the random $\zeta$-free $\eta$-hypergraphs.
\end{definition}

	For fixed $m > k \geq 2$ and $\eta$ such that $\eta(n) = 1$ iff $n = k$, and $\zeta$ such that $\zeta(k) = m$ and $\zeta(n) = 0$ for $n \neq k$, the structure $M(\eta, \zeta)$ is simply the random $m$-free $k$-uniform hypergraph \cite{hru}. Notice that for $\zeta \in \omega^\omega$ such that $\zeta(n) = \eta(n)n$ we have $M(\eta, \zeta) \cong M(\eta)$, and so this setting actually generalizes the previous.	
	Finally, also in this case $\mathbf{K}(\eta, \zeta)$ is closed under free amalgamation, and so by our Corollary \ref{SSIP_free} we have:
	
	\begin{corollary} For $\eta, \zeta$ \mbox{as above, $M(\eta, \zeta)$ has the strong small index property.}
\end{corollary}

		We now re-state the main result of \cite{Sh_Pa_recon} (actually a more informative version of that result already stated in \cite{Sh_Pa_recon}) and use it to prove Theorem \ref{pairwise_non_iso}.
		
		\begin{fact}[\cite{Sh_Pa_recon}]\label{reconstruction} Let $M$ and $N$ be countable $\aleph_0$-categorical structures with the strong small index property and no algebraicity. Then $ Aut(M)$ and $Aut(N)$ are isomorphic as abstract groups if and only if $M$ and $N$ are bi-definable. Moreover, if $\pi: Aut(M) \cong Aut(N)$ is an abstract group isomorphism, then there is a bijection $f: M \rightarrow N$ witnessing the bi-definability of $M$ and $N$ such that $\pi(\alpha) = f \alpha f^{-1}$.	
\end{fact}
		
		\begin{definition}\label{bidef} We say that two structures $M$ and $N$ are bi-definable if there is a bijection $f:M \rightarrow N$ such that for every $A \subseteq M^n$, $A$ is $\emptyset$-definable in $M$ if and only if $f(A)$ is $\emptyset$-definable in $N$.
\end{definition}

	\begin{proof}[Proof of Theorem \ref{pairwise_non_iso}] As noticed after Definition \ref{iper_2}, every random $\eta$-hypergraph can be considered as a $\zeta$-free $\eta$-hypergraph for appropriate $\zeta$, and so it suffices to deal with the class of $\zeta$-free $\eta$-hypergraphs. Let $M_1 = M(\eta_1, \zeta_1)$ and $M_2 = M(\eta_2, \zeta_2)$ and suppose that $\pi: Aut(M_1) \cong Aut(M_2)$. We will show that $(\eta_1, \zeta_1) = (\eta_2, \zeta_2)$. By Fact \ref{reconstruction}, there is $f: M_1 \rightarrow M_2$ witnessing bi-definability and inducing $\pi$. Without loss of generality $f = id_{M_1}$, and so $Aut(M_1) = Aut(M_2) := H$. Denote the domain of $M_{1}$ as $A$ (which is the same as the domain of $M_2$). First of all we show that $\eta_1 = \eta_2$. Notice that for every $1 < n < \omega$ the following assertions are equivalent:
	\begin{enumerate}[$(a)_n$]
	\item $\eta_{\ell} (n) = 1$ ($\ell = 1, 2$);
	\item there are $\bar{a} \neq \bar{b} \in A^n$ each with no repetitions such that:
	\begin{enumerate}[(i)]
	\item for every $h \in H$, $h(\bar{a}) \neq h(\bar{b})$;
	\item for every $i < n$ there exists $h \in H$ such that $\bigwedge_{i \neq j < n} h(a_j) = b_j$.
	\end{enumerate}
\end{enumerate}
Since the statement $(b)_n$ does not depend on $\ell$ we conclude that $\eta_1 = \eta_2$. Suppose now that $\eta_1 = \eta_2$, we show that $\zeta_1 = \zeta_2$. Given $n \leq k < \omega$ we denote by $[k]^n$ the set of subsets of $\{ 1, ..., k \}$ of size $n$. Notice that for every $1 < n \leq k < \omega$ the following assertions are equivalent:
	\begin{enumerate}[$(a')_{n, k}$]
	\item $k < \zeta_{\ell} (n)$ ($\ell = 1, 2$);
	\item for every $\{ \bar{a}_x = (a^x_{x(1)}, ..., a^x_{x(n)}) : x = \{ x(1) < \cdots < x(n) \} \in [k]^n \}$ each with no repetitions if (i') then (ii'), where:
	\begin{enumerate}[(i')]
	\item for every $x \neq y \in [k]^n$ with $|x \cap y| > 0$ there exists $h_{x, y} \in H$ such that $\bigwedge_{p \in x \cap y} h_{x, y}(a^{x}_{p}) = a^{y}_{p}$;
	\item there exists $\bar{a}_* = (a^*_1, ..., a^*_k) \in A^k$ with no repetitions and $h_x \in H$, $x \in [k]^n$, such that $\bigwedge_{i < n}h_x(a^x_{x(i)}) = a^*_{x(i)}$.
	\end{enumerate}
\end{enumerate}
Since the statement $(b)_{n, k}$ does not depend on $\ell$ we conclude that $\zeta_1 = \zeta_2$.
\end{proof}

\end{document}